\documentclass[12pt,oneside,reqno]{amsart}
\usepackage[all]{xy}
\usepackage{amsfonts,amsmath,oldgerm,amssymb,amscd, comment,multirow, color}
\UseComputerModernTips
\numberwithin{equation}{section}
\usepackage[breaklinks]{hyperref}
\allowdisplaybreaks

\usepackage{enumerate}

\usepackage{caption} 
\newtheorem{theorem}{Theorem}[section]

\newtheorem{lemma}[theorem]{{\bf Lemma}}
\newtheorem{coro}[theorem]{{\bf Corollary}}
\newtheorem{definition}[theorem]{Definition}

\newtheorem{remark}[subsection]{Remark}

\newcommand{\Z}{\mbox{$\mathbb Z$}}

\begin{document}

	\title[ Arithmetic density and congruences of $\ell$-regular bipartitions ]{ Arithmetic density and congruences of $\ell$-regular bipartitions} 
	
	\author[N.K. Meher*]{N.K. Meher}
	\address{Nabin Kumar Meher, Department of Mathematics, Indian Institute of Information Raichur, Govt. Engineering College Campus, Yermarus, Raichur, Karnataka, India 584135.}
	\email{mehernabin@gmail.com, nabinmeher@iiitr.ac.in}
	
	\author[ Ankita Jindal]{ Ankita Jindal$^{*}$}
	\address{Ankita Jindal, Indian Statistical Institute Bangalore, 8th Mile, Mysore Road, Bangalore, Karnataka, India 560059}
	\email{ankitajindal1203@gmail.com }

	\thanks{
		$*$ is the corresponding author.}
	\maketitle
	\pagenumbering{arabic}
	\pagestyle{headings}
	\begin{abstract}
		Let $ B_{\ell}(n)$ denote the number of $\ell$-regular bipartitions of $n.$ In this article, we prove that $ B_{\ell}(n)$ is always almost divisible by $p_i^j$ if $p_i^{2a_i}\geq \ell,$ where $j$ is a fixed positive integer and $\ell=p_1^{a_1}p_2^{a_2}\ldots p_m^{a_m},$ where $p_i$ are prime numbers $\geq 5.$ Further, we obtain an infinities families of congruences for $B_3(n)$ and $B_5(n)$ by using Hecke eigen form theory and a result of Newman \cite{Newmann1959}. Furthermore, by applying Radu and Seller's approach, we obtain an algorithm from which we get several congruences for $B_{p}(n)$, where $p$ is a prime number. 
		\end{abstract}
	\maketitle
	
	\section{Introduction}
	For any positive integer $\ell >1,$ a partition of a positive integer is called an $\ell$-regular partition if none of its parts is divisible by $\ell.$ Let $b_{\ell}(n)$ denotes the number of $\ell$-regular partitions of $n.$ The generating function for $b_{\ell}(n)$ is given by 
	$$ \sum_{n=0}^{\infty} b_{\ell}(n) q^n= \frac{f_{\ell}}{f_1},$$
	where $f_{\ell}$ is defined by $f_{\ell}= \prod_{m=1}^{\infty} (1- q^{\ell m}).$
	
	Gordan and Ono \cite{Gordan1997} proved a density result on $\ell$-regular partition function $b_{\ell}(n)$. In particular, they proved that if $p$ is a prime number and $p^{\hbox{ord}_p(\ell)} \geq \sqrt{\ell},$ then for any positive integer $j,$ the arithmetic density of positive integers $n$ such that $b_{\ell}(n) \equiv 0 \pmod{p^j}$ is one. Andrews et al. \cite{Andrews2010} found infinite families of congruences modulo $3$ for $b_4(n).$
	
	A $(k,\ell)$-regular bipartition of $ n$ is a bipartition $(\lambda, \mu)$ of $n$ such that $ \lambda$ is a $k$-regular partition and $\mu$ is an $\ell$-regular partition. 
	An $\ell$-regular bipartition of $n$ is an ordered pair of $\ell$-regular partitions $(\lambda, \mu)$ such that the sum of all of the parts equals $n.$ Let $B_{\ell}(n)$ denote the number of $\ell$-regular bipartitions of $n.$ The generating function of $B_{\ell}(n)$ satisfies 
	
	\begin{align}\label{eq501}
	 \sum_{n=0}^{\infty} B_{\ell}(n) q^n= \frac{f^2_{\ell}}{f_1^2}.
	 \end{align}
 
 In this article, we study the arithmetic density of $B_{\ell}(n)$ and prove new results similar to that of Gordan and Ono for $b_{\ell}(n)$ (see \cite{Gordan1997}).
 
 
 \begin{theorem}\label{mainthm11} For a positive integer $m,$ let $a_1, a_2, \ldots, a_m$ be positive integers. Let $\ell=p_1^{a_1}p_2^{a_2}\ldots p_m^{a_m},$ where $p_i$ are prime numbers $\geq 5.$ If $p_i^{2a_i} \geq \ell,$ then for every positive integer $j,$  the set 
 	$$\left\{0 < n \leq  X: B_{\ell}(n)\equiv 0 \pmod{p_i^j} \right\}$$ 
 has arithmetic density $1.$
 \end{theorem}

 We have the following immediate result as a particular case of the above theorem.
 \begin{coro}\label{coro11}
 	Let $j$ be a positive integer and $ p \geq 5$ be a prime number. Then $B_p(n)$ is almost always divisible by $p^j,$ namely
 	$$ \lim\limits_{X \to \infty} \frac{\# \left\{0 < n \leq X: B_p(n) \equiv 0 \pmod{p^j}\right\}}{X} =1.$$ 	
 \end{coro}


 In 2017, Xia and Yao \cite{Xia2017} studied $(k,\ell)$-regular bipartitions and proved that for a postive integer $t$ and a prime $p\geq 5$ with $\left( \frac{-t}{p}\right)= -1,$ we have
\begin{equation}\label{eq001}
 	B_{3,t} \left( p^{2\alpha +1}n+ \frac{(1+t)(p^{2\alpha+2}-1)}{24}\right) \equiv 0 \pmod3
 \end{equation}
 whenever $n, \alpha \geq 0$ with $p \nmid n$. Note that $\left( \frac{-3}{p}\right)= -1$ for a prime $p$ with $p \equiv 5\pmod6$. Hence by taking $t=3$ and prime $p$ with $p \equiv 5\pmod6$, in \eqref{eq001} implies that
 \begin{equation}\label{eq003}
 	B_{3} \left( p^{2\alpha +1}n+ \frac{(p^{2\alpha+2}-1)}{6}\right) \equiv 0 \pmod3
 \end{equation}
 whenever $n, \alpha \geq 0$ with $p \nmid n.$
 
In \cite{Xia2017},  Xia and Yao also proved that for a postive integer $t$ and a prime $p\geq 5$ with $\left( \frac{-3t}{p}\right)= -1,$
 \begin{equation}\label{eq002}
 	B_{5,t} \left( p^{2\alpha +1}n+ \frac{(3+t)(p^{2\alpha+2}-1)}{24}\right) \equiv 0 \pmod5
 \end{equation}
  whenever $n, \alpha \geq 0$ with $p \nmid n$. When we take $t=5$ and prime $ p >5$ with $p \equiv 2\pmod3$ and $\left( \frac{5}{p}\right)= 1,$ then \eqref{eq002} implies that
  \begin{equation}\label{eq004}
  	B_{5} \left( p^{2\alpha +1}n+ \frac{(p^{2\alpha+2}-1)}{3}\right) \equiv 0 \pmod5
  \end{equation} 
whenever $n, \alpha \geq 0$ with $p \nmid n.$

Next, we give new congruence equations for $3$-regular and $5$-regular bipartition functions $B_3(n)$ and $B_5(n)$ which generalises \eqref{eq003} and \eqref{eq004}. We first prove the following infinite families of congruences for $ B_3(n)$ modulo $3$ using the theory of Hecke-eigenforms.
	\begin{theorem}\label{thm8}
	Let $k$ and $n$ be non-negative integers. Let $p_1, p_2,\ldots,p_{k+1}$ be primes such that $p_i \geq 5$ and $p_i \not \equiv 1 \pmod6$ for each $1 \leq i \leq k+1$. For any integer $j \not \equiv 0 \pmod {p_{k+1}},$ we have
	$$B_{3} \left( p_1^2 p_2^2 \cdots p_k^2 p_{k+1}^2 n + \frac{p_1^2 p_2^2 \cdots p_k^2 p_{k+1}\left(6j+ p_{k+1}\right)-1}{6} \right) \equiv 0 \pmod3.$$
\end{theorem}
By taking $p_1= p_2= \cdots= p_{k+1}= p$ in Theorem \ref{thm8}, we obtain the following corollary. 
\begin{coro}\label{coro3}
	Let $k$ and $n$ be non-negative integers. Let $p \geq 5$ be a prime with $p \equiv 5 \pmod6$. For any integer $j \not \equiv 0 \pmod p,$ we have
	$$B_{3}\left(p^{2k+2}n +p^{2k+1}j+ \frac{p^{2k+2}-1}{6} \right) \equiv 0 \pmod 3.$$
\end{coro}
For instance, if we substitute $p=5,$ $j\not \equiv 0 \pmod 5$  and $k=0$ in Corollary \ref{coro3}, then we get
	$$B_{3}\left(25n +5j+ 4 \right) \equiv 0 \pmod 3.$$	
\begin{remark}
Note that \eqref{eq003} can be obtained as a particular case of Corollary \ref{coro3} by taking $n=0$. 
\end{remark}

 Next, by applying an identity due to Newman \cite{Newmann1959}, we obtain the following theorem.
\begin{theorem}\label{Newmann1}
Let $k$ be a non-negative integer. Let $p$ be a prime number with $p \equiv 1 \pmod 6.$ If $B_{3}\left(\frac{p-1}{6}\right) \equiv 0 \pmod 3,$ then for $n \geq 0$ satisfying $p \nmid (6n+1)$, we have 
	\begin{equation}\label{eq49}
		B_{3}\left(p^{2k+1}n+ \frac{p^{2k +1}-1}{6}\right) \equiv 0 \pmod 3.
	\end{equation}
\end{theorem}


	Furthermore, we prove the following multiplicative formulae for $3$-regular bipartitions modulo $3$.
	\begin{theorem}\label{thm9}
		Let $k$ be a positive integer and $p$ be a prime number such that $ p \equiv 5 \pmod 6.$ Let $r$ be a non-negative integer such that $p$ divides $6r+5$. We have
		\begin{align*}
			B_{3}\left(p^{k+1}n+pr+ \frac{5p-1}{6}\right)\equiv  (-p) \cdot  B_{3}\left(p^{k-1}n+ \frac{6r+5-p}{6p} \right) \pmod3.
		\end{align*}
	\end{theorem}
	\begin{coro}\label{coro4}
		Let $k$ be a positive integer and $p$ be a prime number such that $p \equiv 5 \pmod 6$. For $n\geq 0$, we have
		\begin{align*}
			B_{3}\left(p^{2k}n+ \frac{p^{2k}-1}{6} \right)&\equiv  \left(-p\right)^k \cdot  B_{3}(n) \pmod3.
		\end{align*}
	\end{coro}
 Next, we prove the following Ramanujan type congruences for $B_5(n).$ We obtain the following infinite families of congruences for $B_5(n)$ by using Hecke eigenform theory.
		\begin{theorem}\label{thm10}
	Let $k$ and $n$ be non-negative integers. Let $p_1, p_2,\ldots,p_{k+1}$ be primes such that $p_i \equiv 2 \pmod3$ for each $1 \leq i \leq k+1$. For any integer $j \not \equiv 0 \pmod {p_{k+1}},$ we have
	$$B_{5} \left( p_1^2 p_2^2 \cdots p_k^2 p_{k+1}^2 n + \frac{p_1^2 p_2^2 \cdots p_k^2 p_{k+1}\left(3j+ p_{k+1}\right)-1}{3} \right) \equiv 0 \pmod5.$$
\end{theorem}
By taking $p_1= p_2= \cdots= p_{k+1}= p$ in Theorem \ref{thm10}, we get the following corollary. 
\begin{coro}\label{coro5}
	Let $k$ and $n$ be a non-negative integers. Let $p$ be a prime with $p \equiv 2 \pmod3.$  For any integer $j \not \equiv 0 \pmod p,$ we have
	$$B_{5}\left(p^{2k+2}n +p^{2k+1}j+ \frac{p^{2k+2}-1}{3} \right) \equiv 0 \pmod 5.$$
\end{coro}

\begin{remark}
	Note that \eqref{eq004} can be obtained as a particular case of Corollary \ref{coro5} by taking $n=0$. 
\end{remark}

\begin{remark}
In Corollary \ref{coro5}, replacing $k$ by $k-1$ and putting $p=2, j=1$ we get
$$B_{5}\left(2^{2k}n +2^{2k-1}+ \frac{2^{2k}-1}{3} \right) \equiv B_{5}\left( 4^{k}n +\frac{5 \cdot 4^{k}-2}{6} \right) \equiv 0 \pmod 5,$$ 
which is a result of T. Kathiravan and S. N. Fathima \cite[Theorem $1.1$]{Kathiravan2017}.
\end{remark}

  Furthermore, we prove some infinite families of congruences modulo $5$ for $B_5(n)$ by using a result due to Newman \cite{Newmann1959}.
\begin{theorem}\label{Newmann2}
Let $k$ be a non-negative integer. Let $p$ be a prime number with $p \equiv 1 \pmod 6.$ If $B_{5}\left(\frac{p-1}{3}\right) \equiv 0 \pmod 5,$ then for any integer $n \geq 0$ satisfying $p \nmid (3n+1)$, we have
	\begin{equation}\label{eq149}
		B_{5}\left(p^{2k+1}n+ \frac{p^{2k +1}-1}{3}\right) \equiv 0 \pmod 5.
	\end{equation}
\end{theorem}

Define the set 
\begin{equation}\label{eq800bcc}
	\mathbb{\bar{S}}:= \left\{ 103,  157, 193\right\}.
\end{equation}
It is easy to check that for any $p \in  \mathbb{\bar{S}},$ $ B_5\left(\frac{p-1}{3}\right) \equiv 0 \pmod 5.$
\begin{coro}
	Let $\mathbb{\bar{S}}$ be defined by \eqref{eq800bcc}. If $p \in \mathbb{\bar{S}}$ and $ p \nmid (3n+1),$ then for $n \geq 0, $ and $k \geq 0,$ we have
	\begin{equation}\label{eq49a}
		B_{5}\left(p^{2 k+1}n+ \frac{p^{2 k +1}-1}{3}\right) \equiv 0 \pmod 5.
	\end{equation}
\end{coro}

	Furthermore, we prove the following multiplicative formulae for $5$-regular bipartition function modulo $5$.
\begin{theorem}\label{thm11}
	Let $k$ be a positive integer and $p$ be a prime number such that $ p \equiv 2 \pmod 3.$ Let $r$ be a non-negative integer such that $p$ divides $3r+2,$ then
	\begin{align*}
		B_{5}\left(p^{k+1}n+pr+ \frac{2p-1}{3}\right)\equiv  (-p^3) \cdot  B_{5}\left(p^{k-1}n+ \frac{3r+2-p}{3p} \right) \pmod5.
	\end{align*}
\end{theorem}
\begin{coro}\label{coro6}
	Let $k$ be a positive integer and $p$ be a prime number such that $p \equiv 2 \pmod 3.$ Then
	\begin{align*}
		B_{5}\left(p^{2k}n+ \frac{p^{2k}-1}{3} \right)&\equiv  \left(-p^3\right)^k \cdot  B_{5}(n) \pmod5.
	\end{align*}
\end{coro}

\begin{remark}
	In particular, if we take $p=2$ in Corollary \ref{coro6}, we get 
	\begin{align*}
		B_{5}\left(4^{k}n+ \frac{4^{k}-1}{3} \right)&\equiv  \left(-2^3\right)^k \cdot  B_{5}(n) \equiv  2^k \cdot  B_{5}(n) \pmod5,
	\end{align*}
	which was also obtained by T. Kathiravan and S. N. Fathima \cite[Theorem $1.1$]{Kathiravan2017}.
\end{remark}

In this article, we also employ a method of Radu and Seller \cite{Radu2011} and we obtain an algorithm for congruences for $B_p(n)$ where $p$ is a prime number. Before stating our result in this direction, we define some notation. For an integer $m$, a prime $p\geqslant 5$ and $t\in \{0,1,\dots,m-1\}$, we set
\begin{align*}
	\kappa:=&\kappa(m)=\gcd(m^2-1,24),\\
	\hat{p}:=& \frac{p^2-1}{24},\\
	A_t:=&A_t(m,p)=\frac{12m}{\gcd(-\kappa(12 t+p-1),12m)}, \\
	\epsilon_p:=&\epsilon_p(m,p)=\begin{cases}
		1 & \textrm{ if } p\nmid m,\\
		0 & \textrm{ if } p| m.
	\end{cases}
\end{align*}
Note that $\hat{p}\in \Z$ and $A_t \in \Z$ for each $t\in \{0,1,\dots,m-1\}$.

\begin{theorem}\label{mainthmRaduSellers}
	Let $p\geq 5$ be a prime and let $u$ be an integer. For a positive integer $g$, let $e_1, e_2, \dots, e_g$ be positive integers. Let $m=p_1^{e_1} p_2^{e_2} \cdots p_g^{e_g}$ where $p_i\geq 5$'s are prime numbers. Let $t \in \{0,1,\ldots, m-1\}$ such that $A_t$ divides $p^{\epsilon_p}p_1p_2\cdots p_{g}$. Define
	\begin{align*}
		P(t):=\left\{t^{'}: \exists [s]_{24m} \textrm{ such that } t^{'}\equiv ts+\frac{(s-1)(p-1)}{12}  \pmod {m}\right\}
	\end{align*} 
	where $[s]_{24m}$ is the residue class of $s$ in $\mathbb{Z}_{24m}$. If the congruence $B_p(mn+t^{'}) \equiv 0 \pmod u$ holds for all $t^{'} \in P(t)$ and $ 0\leq n \leq \left \lfloor \frac{m(p-1)}{12}\left((p+1)^{\epsilon_p}(p_1+1)(p_2+1)\cdots (p_{g}+1)-p\right) - \frac{p-1}{12m} \right \rfloor$, then $B_p(mn+t^{'}) \equiv 0 \pmod u$ holds for all $t^{'} \in P(t)$ and $n \geq 0$.
\end{theorem}
	
	
	\section{Preliminaries}
	We recall some basic facts and definition on modular forms. For more details, one can see \cite{Koblitz}, \cite{Ono2004}. We start with some matrix groups. We define
	\begin{align*}
		\Gamma:=\mathrm{SL_2}(\mathbb{Z})= &\left\{ \begin{bmatrix}
			a && b \\c && d
		\end{bmatrix}: a, b, c, d \in \mathbb{Z}, ad-bc=1 \right\},\\
		\Gamma_{\infty}:= &\left\{\begin{bmatrix}
			1 &n\\ 0&1	\end{bmatrix}: n \in \mathbb{Z}\right\}.
	\end{align*}
	Let $N$ be a positive integer.  Then $ \Gamma_{0}(N), \Gamma_{1}(N)  $ and $ \Gamma(N) $ are defined as
	\begin{align*}
		\Gamma_{0}(N):=& \left\{ \begin{bmatrix}
			a && b \\c && d
		\end{bmatrix} \in \mathrm{SL_2}(\mathbb{Z}) : c\equiv0 \pmod N \right\},\\
		\Gamma_{1}(N):=& \left\{ \begin{bmatrix}
			a && b \\c && d
		\end{bmatrix} \in \Gamma_{0}(N) : a\equiv d  \equiv 1 \pmod N \right\}
	\end{align*}
	and 
	\begin{align*}
		\Gamma(N):= \left\{ \begin{bmatrix}
			a && b \\c && d
		\end{bmatrix} \in \mathrm{SL_2}(\mathbb{Z}) : a\equiv d  \equiv 1 \pmod N,  b \equiv c  \equiv 0 \pmod N \right\}.
	\end{align*}
	A subgroup of $\Gamma=\mathrm{SL_2}(\mathbb{Z})$ is called a congruence subgroup if it contains $ \Gamma(N)$ for some positive integer $N$ and the smallest $N$ with this property is called its level. Note that $ \Gamma_{0}(N)$ and $ \Gamma_{1}(N)$ are congruence subgroup of level $N,$ whereas $ \mathrm{SL_2}(\mathbb{Z}) $ and $\Gamma_{\infty}$ are congruence subgroups of level $1.$ The index of $\Gamma_0(N)$ in $\Gamma$ is 
	\begin{align*}
		[\Gamma:\Gamma_0(N)]=N\prod\limits_{p|N}\left(1+\frac 1p\right)
	\end{align*}
	where $p$ runs over the prime divisors of $N$.
	
	Let $\mathbb{H}$ denote the upper half of the complex plane $\mathbb{C}$. The group 
	\begin{align*}
		\mathrm{GL_2^{+}}(\mathbb{R}):= \left\{ \begin{bmatrix}
			a && b \\c && d
		\end{bmatrix}: a, b, c, d \in \mathbb{R}, ad-bc>0 \right\},
	\end{align*}
	acts on $\mathbb{H}$ by $ \begin{bmatrix}
		a && b \\c && d
	\end{bmatrix} z = \frac{az+b}{cz+d}.$ We identify $\infty$ with $\frac{1}{0}$ and define $ \begin{bmatrix}
		a && b \\c && d
	\end{bmatrix} \frac{r}{s} = \frac{ar+bs}{cr+ds},$ where $\frac{r}{s} \in \mathbb{Q} \cup \{ \infty\}$. This gives an action of $\mathrm{GL_2^{+}}(\mathbb{R})$ on the extended half plane $\mathbb{H}^{*}=\mathbb{H} \cup \mathbb{Q} \cup \{\infty\}$. Suppose that $\Gamma$ is a congruence subgroup of $\mathrm{SL_2}(\mathbb{Z})$. A cusp of $\Gamma$ is an equivalence class in $\mathbb{P}^{1}=\mathbb{Q} \cup \{\infty\}$ under the action of $\Gamma$.
	
	The group $\mathrm{GL_2^{+}}(\mathbb{R})$ also acts on functions $g:\mathbb{H} \rightarrow \mathbb{C}$. In particular, suppose that $\gamma=\begin{bmatrix}
		a && b \\c && d
	\end{bmatrix}\in \mathrm{GL_2^{+}}(\mathbb{R})$. If $g(z)$ is a meromorphic function on $\mathbb{H}$ and $k$ is an integer, then define the slash operator $|_{k}$ by
	\begin{align*}
		(g|_{k} \gamma)(z):= (\det \gamma)^{k/2} (cz+d)^{-k} g(\gamma z).
	\end{align*}
	
	\begin{definition}
		Let $\Gamma$ be a congruence subgroup of level $N$. A holomorphic function $g:\mathbb{H} \rightarrow \mathbb{C}$ is called a modular form of integer weight $k$ on $\Gamma$ if the following hold:
		\begin{enumerate}[$(1)$]
			\item We have
			\begin{align*}
				g \left( \frac{az+b}{cz+d}\right)=(cz+d)^{k} g(z)
			\end{align*}
			for all $z \in \mathbb{H}$ and $\begin{bmatrix}
				a && b \\c && d
			\end{bmatrix}\in \Gamma$. 
			\item If $\gamma\in SL_2 (\mathbb{Z})$, then $(g|_{k} \gamma)(z)$ has a Fourier expnasion of the form
			\begin{align*}
				(g|_{k} \gamma)(z):= \sum \limits_{n\geq 0}a_{\gamma}(n) q_N^{n}
			\end{align*}
			where $q_N:=e^{2\pi i z /N}$.
		\end{enumerate}
	\end{definition}
	For a positive integer $k$, the complex vector space of modular forms of weight $k$ with respect to a congruence subgroup $\Gamma$ is denoted by $M_{k}(\Gamma)$.
	
	\begin{definition} \cite[Definition 1.15]{Ono2004}
		If $\chi$ is a Dirichlet character modulo $N$, then we say that a modular form $g \in M_{k}(\Gamma_1(N))$ has Nebentypus character $\chi$ if 
		\begin{align*}
			g \left( \frac{az+b}{cz+d}\right)=\chi(d) (cz+d)^{k} g(z)
		\end{align*}
		for all $z \in \mathbb{H}$ and $\begin{bmatrix}
			a && b \\c && d
		\end{bmatrix}\in \Gamma_{0}(N)$. The space of such modular forms is denoted by $M_{k}(\Gamma_0(N), \chi)$.
	\end{definition}
	
	The relevant modular forms for the results obtained in this article arise from eta-quotients. Recall that the Dedekind eta-function $\eta (z)$ is defined by 
	\begin{align*}
		\eta (z):= q^{1/24}(q;q)_{\infty}=q^{1/24} \prod\limits_{n=1}^{\infty} (1-q^n)
	\end{align*}
	where $q:=e^{2\pi i z}$ and $z \in \mathbb{H}$. A function $g(z)$ is called an eta-quotient if it is of the form
	\begin{align*}
		g(z):= \prod\limits_{\delta|N} \eta(\delta z)^{r_{\delta}}
	\end{align*}
	where $N$ and $r_{\delta}$ are integers with $N>0$. 
	
	\begin{theorem} \cite[Theorem 1.64]{Ono2004} \label{thm2.3}
		If $g(z)=\prod\limits_{\delta|N} \eta(\delta z)^{r_{\delta}}$ is an eta-quotient such that $k= \frac 12$ $\sum_{\delta|N} r_{\delta}\in \mathbb{Z}$, 
		\begin{align*}
			\sum\limits_{\delta|N} \delta r_{\delta} \equiv 0\pmod {24}	\quad \textrm{and} \quad \sum\limits_{\delta|N} \frac{N}{\delta}r_{\delta} \equiv 0\pmod {24},
		\end{align*}
		then $g(z)$ satisfies
		\begin{align*}
			g \left( \frac{az+b}{cz+d}\right)=\chi(d) (cz+d)^{k} f(z)
		\end{align*}
		for each $\begin{bmatrix}
			a && b \\c && d
		\end{bmatrix}\in \Gamma_{0}(N)$. Here the character $\chi$ is defined by $\chi(d):= \left(\frac{(-1)^{k}s}{d}\right)$ where $s=\prod_{\delta|N} \delta ^{r_{\delta}}$.
	\end{theorem}
	
	\begin{theorem} \cite[Theorem 1.65]{Ono2004} \label{thm2.4}
		Let $c,d$ and $N$ be positive integers with $d|N$ and $\gcd(c,d)=1$. If $f$ is an eta-quotient satisfying the conditions of Theorem \ref{thm2.3} for $N$, then the order of vanishing of $f(z)$ at the cusp $\frac{c}{d}$ is
		\begin{align*}
			\frac{N}{24}\sum\limits_{\delta|N} \frac{\gcd(d, \delta)^2 r_{\delta}}{\gcd(d, \frac{N}{ d} )d \delta}.
		\end{align*}
	\end{theorem}
	Suppose that $g(z)$ is an eta-quotient satisfying the conditions of Theorem \ref{thm2.3} and that the associated weight $k$ is a positive integer. If $g(z)$ is holomorphic at all of the cusps of $\Gamma_0(N)$, then $g(z) \in M_{k}(\Gamma_0(N), \chi)$. Theorem \ref{thm2.4} gives the necessary criterion for determining orders of an eta-quotient at cusps. In the proofs of our results, we use Theorems \ref{thm2.3} and \ref{thm2.4} to prove that $g(z) \in M_{k}(\Gamma_0(N), \chi)$ for certain eta-quotients $g(z)$ we consider in the sequel.
	
	We shall now mention a result of Serre \cite[P. 43]{Serre1974} which will be used later. 
	
	\begin{theorem}\label{thm2.5}
		Let $g(z) \in M_{k}(\Gamma_{0}(N), \chi)$ has Fourier expansion
		$$ g(z)= \sum_{n=0}^{\infty} b(n) q^n \in \mathbb{Z}[[q]].$$
		Then for a positive integer $r$,  there is a constant $\alpha>0$ such that
		$$\#\{ 0 < n \leq X: b(n) \not \equiv 0 \pmod{r}\} = \mathcal{O}\left( \frac{X}{(\log X)^{\alpha}}\right).$$
		Equivalently 
		\begin{align}\label{2e1}
			\begin{split}
				\lim\limits_{X \to \infty} \frac{\#\{ 0 < n \leq X: b(n) \not \equiv 0 \pmod{r}\}}{X}= 0.
		\end{split}	\end{align}
	\end{theorem}

	We finally recall the definition of Hecke operators and a few relavent results. Let $m$ be a positive integer and $g(z)= \sum \limits_{n= 0}^{\infty}a(n) q^{n}\in M_{k}(\Gamma_0(N), \chi)$. Then the action of Hecke operator $T_m$ on $f(z)$ is defined by
	\begin{align*}
		g(z)|T_{m} := \sum \limits_{n= 0}^{\infty} \left(\sum \limits_{d|\gcd(n,m)} \chi(d) d^{k-1} a\left(\frac{mn}{d^2}\right)\right)q^{n}.
	\end{align*}
	In particular, if $m=p$ is a prime, we have
	\begin{align*}
		g(z)|T_p := \sum \limits_{n= 0}^{\infty}\left( a(pn) + \chi(p) p^{k-1} a\left(\frac{n}{p}\right)\right)q^{n}.
	\end{align*}
	We note that $a(n)=0$ unless $n$ is a non-negative integer.

		\section{Proof of Theorem \ref{mainthm11}}
		Let $\ell=p_1^{a_1}p_2^{a_2}\ldots p_m^{a_m},$ where $p_i$'s
		are primes.
		From \eqref{eq501}, we get
		\begin{small}
			\begin{equation}\label{eq701}
				\sum_{n=0}^{\infty} B_{\ell}(n) q^n =  \frac{(q^{\ell};q^{\ell})^{2}_{\infty}}{(q;q)^2_{\infty}}.
			\end{equation}
		\end{small}
		Note that for any prime $p$ and positive integers $j, k$  we have
		\begin{small}
			\begin{equation}\label{eq702}
				(q;q)_{\infty}^{p^j}\equiv (q^p;q^p)_{\infty}^{p^{j-1}} \pmod{p^j} \implies 	(q^k;q^k)_{\infty}^{p^j}\equiv (q^{kp};q^{kp})_{\infty}^{p^{j-1}} \pmod{p^j}.
			\end{equation}
		\end{small}
		For a positive integer $i,$ we define $$A_i(z):= \frac{\eta(24z)^{p_i^{a_i}}}{\eta \left(24p_i^{a_i}z\right)} .$$
		Using \eqref{eq702}, we get 
		\begin{small}
			\begin{equation*}
				A_i^{p_i^j}(z):= \frac{\eta(24z)^{p_i^{a_i+j}}}{\eta \left(24p_i^{a_i}z\right)^{p_i^j}} \equiv 1 \pmod{p_i^{j+1}}.
			\end{equation*}
		\end{small}
		Define 
		$$B_{i,j,\ell}(z)= \left(\frac{\eta^2(24\ell z)}{\eta^2(24z)}\right)A_i^{p_i^j}(z) =\frac{\eta^{2}(24\ell z) \eta^{(p_i^{a_i+j}-2)}(24z)}{\eta^{p_i^j} \left(24p_i^{a_i}z\right)} .$$
		
		On modulo $p_i^{j+1},$ we get
		\begin{small}
			\begin{equation}\label{eq703}
				B_{i,j,\ell}(z)= \frac{\eta(24 \ell z)^{2}}{\eta^2(24z)}= q^{2\ell-2} \frac{(q^{24\ell};q^{24\ell})^{2}_{\infty}}{(q^{24};q^{24})^2_{\infty}}.
			\end{equation}
		\end{small}
		Combining \eqref{eq701} and \eqref{eq703} together, we obtain
		\begin{small}
			\begin{equation}\label{eq704}
				B_{i,j,\ell}(z)\equiv  q^{2\ell-2} \frac{(q^{24\ell};q^{24\ell})^{2}_{\infty}}{(q^{24};q^{24})^2_{\infty}} \equiv \sum_{n=0}^{\infty} B_{\ell}(n)q^{24n+2\ell-2} \pmod{p_i^{j+1}}.
			\end{equation}
		\end{small}
	
		Next, we prove that $B_{i,j,\ell}(z)$ is a modular form. Applying Theorem \ref{thm2.3}, we first estimate the level of eta quotient $B_{i,j,\ell}(z)$ . The level of $ B_{i,j,\ell}(z) $ is $N=24p_1^{a_1}p_2^{a_2}\ldots p_m^{a_m} M,$ where $M$ is the smallest positive integer which satisfies 
		\begin{small}
			\begin{align*}
				24 \ell M \left[ \frac{2}{24\ell} + \frac{p_i^{a_i+j}-2  }{2^3 3} - \frac{p_i^{j}}{24p_i^{a_i}} \right]\equiv 0 \pmod{24} 
				\implies 2M\left[ 1- \ell + \ell  p_i^{j}   \frac{\left(p_i^{2a_i}-1 \right) }{2p_i^{a_i}} \right]\equiv 0 \pmod{24}.
			\end{align*}
		\end{small}
	
		Thus, $M=12$ and the level of $B_{i,j,\ell}(z)$  is $N=2^5 3^2 \ell $. The cusps of $\Gamma_{0}(2^5 3^2\ell)$ are given by fractions $\frac{c}{d}$ where $d|2^5 3^2\ell$ and $\gcd(c,d)=1.$ By using Theorem \ref{thm2.4} , we have that $B_{i,j,\ell}(z) $ is holomorphic at a cusp $\frac{c}{d}$ if and only if 
		
			\begin{align}\label{eq502}
				&2 \frac{  \gcd^2(d, 24 \ell)}{24 \ell}+ \left(p_i^{a_i+j}-2\right) \frac{\gcd^2(d, 24)}{24} -p_i^j \frac{\gcd^2(d, 24p_i^{a_i})}{24p_i^{a_i}} \geq 0 \\ \notag
				& \iff L:= 2  + \left(p_i^{a_i+j}-2\right) \ell  G_1 -  \frac{p_i^j}{p_i^{a_i}} \ell  G_2 \geq 0,
			\end{align}
		
		where $G_1= \frac{\gcd^2(d, 24)}{\gcd^2(d, 24 \ell )}$ and  $G_2= \frac{\gcd^2(d, 24 p_i^{a_i})}{\gcd^2(d, 24 \ell)} .$ 
		Let $d$ be a divisor of $ 2^5 3^2 \ell$. We can write $d= 2^{r_1} 3^{r_2} p_i^{k} t $ where $ 0 \leq r_1 \leq 5$,   $ 0 \leq r_2 \leq 2$,  $ 0 \leq k \leq a_i$  and $t|\ell $ but $p_i \nmid t$.
		Next, we find all possible value of divisors of $2^53^2\ell$ to compute equation \eqref{eq502}.
		
		 When $d= 2^{r_1}3^{r_2}t p_i^{k}: 0 \leq r_1\leq 5, 0 \leq r_2\leq 2, t|\ell, \hbox{but} p_i\nmid t, 0 \leq k \leq a_i.$ Then $G_1= \frac{1}{t^2p_i^{2k}}, $ $G_2= \frac{1}{t^2}.$
		Then equation \eqref{eq502} will be
		\begin{align}\label{eq503}
		L= 2  + \left(p_i^{a_i+j}-2\right) \ell \frac{1}{t^2p_i^{2k}}  -  \frac{p_i^j}{p_i^{a_i}} \ell \frac{1}{t^2} \\\label{eq504}
		= 2+ \frac{\ell}{t^2} \left[ p_i^j\left( \frac{{p_i^{a_i}}}{{p_i^{2k}}}- \frac{{1}}{p_i^{a_i}} \right) - \frac{2}{p_i^{2k}} \right].
			\end{align}
		When $k= a_i,$ then $L=  2 \left(1- \frac{\ell}{t^2p_i^{2a_i}}\right) \geq 0$ as $p_i^{2a_i} \geq \ell$ for any $i.$ When $0 \leq k < a_i,$ equation \eqref{eq504} will be non negative if 
		\begin{align*}
		 \left[ p_i^j\left( \frac{{p_i^{a_i}}}{{p_i^{2k}}}- \frac{{1}}{p_i^{a_i}} \right) - \frac{2}{p_i^{2k}} \right] \geq 0.
		\end{align*}
	Note that maximum value of $k$ is $a_i-1.$ 
	Thus, \begin{align*}
		\left[ p_i^j\left( \frac{{p_i^{a_i}}}{{p_i^{2k}}}- \frac{{1}}{p_i^{a_i}} \right) - \frac{2}{p_i^{2k}} \right] \geq 2 \left[ \left( \frac{{p_i^{a_i}}}{{p_i^{2k}}}- \frac{{1}}{p_i^{a_i}} \right) - \frac{1}{p_i^{2k}} \right] = 2 \frac{p_i^{2a_i} - p_i^{2k}- p_i^{a_i} }{p_i^{2k+a_i}}.
	\end{align*}
		 Observe that  $p_i^{2a_i} - p_i^{2k}- p_i^{a_i} \geq p_i^{a_i} \left( p_i^{a_i} -  \frac{p_i^{a_i}}{p_i^2} -1 \right) \geq p_i^{a_i} \left( p_i^{a_i} \left( 1- \frac{1}{p_i^2}\right)  -1 \right) \geq 0.$
		
		Therefore, $B_{i,j,\ell}(z) $ is holomorphic at every cusp $\frac{c}{d}.$
		Using Theorem \ref{thm2.3}, we compute the weight of $B_{i,j,t}(z) $ is $k= \frac{2+ \left(p_i^{a_i+j}-2\right)-p_i^j}{2}=\frac{ p_i^j\left(p_i^{a_i}-1\right)}{2}$ which is a positive integer. The associated character for $B_{i,j,\ell}(z) $ is $$\chi= \left( \frac{(-1)^{k} \left(24 \ell \right)^{2} 24^{\left(p_i^{a_i+j}-2\right)} \left(24p_i^{a_i}\right)^{-p_i^j}}{\bullet}\right).$$ 
		Thus $ B_{i,j,\ell}(z) \in  M_{k}(\Gamma_{0}(N), \chi)$ where $k$, $N$ and $\chi$ are as above. Applying Theorem \ref{thm2.3} and Theorem \ref{thm2.5}, we obtain that the Fourier coefficients of $B_{i,j,\ell}(z) $ satisfies \eqref{2e1} which implies that the Fourier coefficient of $B_{i,j,\ell}(z) $ are almost always divisible by $ p_i^j$. Hence, from $\eqref{eq704}$, we conclude that $B_{\ell}(n)$ are almost always divisible by $p_i^j$. This completes the proof of Theorem \ref{mainthm11}.

	\section{Proof of Congruences for $B_3(n)$}

	\begin{proof}[Proof of Theorem \ref{thm8}] 
		From equation \eqref{eq501}, we have
		\begin{equation}\label{eq800}
			\sum_{n=0}^{\infty} B_{3}(n)q^{n} = \frac{(q^3;q^3)^2_{\infty}}{(q;q)^2_{\infty}}.
		\end{equation}
	Note that for any prime number $p$, we have
	\begin{equation}\label{eq800a}
		(q^p;q^p)_{\infty} \equiv (q;q)^{p}_{\infty} \pmod p.
	\end{equation}
From \eqref{eq800} and \eqref{eq800a}, we get
\begin{equation}\label{eq800b}
	\sum_{n=0}^{\infty} B_{3}(n)q^{n} = \frac{(q^3;q^3)^2_{\infty}}{(q;q)^2_{\infty}} \equiv (q;q)^{4}_{\infty} \pmod 3 .
\end{equation}
Thus, we have
		\begin{equation}\label{eq801}
		\sum_{n=0}^{\infty} B_{3}(n)q^{6n+1} \equiv q (q^6;q^6)^{4}_{\infty}  \equiv \eta^4(6z)\pmod3 .
		\end{equation}
	
		
		By using Theorem \ref{thm2.3}, we obtain $\eta^4(6z) \in S_2(\Gamma_{0}(36), \left(\frac{6^4}{\bullet}\right)). $ Thus $\eta^4(6z) $ has a Fourier expansion i.e.
		\begin{equation}\label{eq802}
			\eta^4(6z)= q-4q^7+ 2 q^{13}+ \cdots= \sum_{n=1}^{\infty}  a(n) q^n .
		\end{equation}
		Thus, $a(n)=0$ if $n \not \equiv 1 \pmod 6,$ for all $n \geq 0.$ From \eqref{eq801} and \eqref{eq802}, comparing the coefficient of $q^{6n+1}$, we get
		\begin{equation}\label{eq803}
			B_{3}(n) \equiv a(6n+1) \pmod 3.
		\end{equation}
		Since $ \eta^4(6z)$ is a Hecke eigenform (see \cite{Martin1996}), it gives
		$$\eta^4(6z)|T_p= \sum_{n=1}^{\infty} \left(a(pn)+ p \cdot  \left(\frac{6^4}{p}\right) a\left(\frac{n}{p}\right)\right) q^n = \lambda(p)\sum_{n=1}^{\infty} a(n)q^n.$$ Note that the Legendre symbol $ \left(\frac{6^4}{p}\right) = 1.$ Comparing the coefficients of $q^n$ on both sides of the above equation, we get
		\begin{equation}\label{eq804}
			a(pn)+ p \cdot a\left(\frac{n}{p}\right) = \lambda(p) a(n).
		\end{equation}
		Since $a(1)=1$ and $a(\frac{1}{p})=0,$ if we put $n=1$ in the above expression, we get $a(p)=\lambda(p).$ As $a(p)=0$ for all $p \not \equiv 1 \pmod 6$ this implies that $\lambda(p)=0$ for all $p \not \equiv 1 \pmod 6.$ From \eqref{eq804} we get that for all $p \not \equiv 1 \pmod 6 $  
		\begin{equation}\label{eq805}
			a(pn)+ p \cdot a\left(\frac{n}{p}\right) =0.
		\end{equation}
		Now, we consider two cases here. If $p \not| n,$ then 
		replacing $n$ by $pn+r$ with $\gcd(r,p)=1$ in \eqref{eq805}, we get
		\begin{equation}\label{eq806}
			a(p^2n+ pr)=0 .
		\end{equation}
	
	Replacing $n$ by $6n-pr+1$ in \eqref{eq806} and using \eqref{eq803}, we get 
		\begin{equation}\label{eq807}
			B_{3}\left(p^2n +pr \frac{(1- p^2)}{6} +  \frac{p^2-1}{6}\right) \equiv 0 \pmod3.
		\end{equation}
		Since $p \equiv 5 \pmod 6,$ we have $6| (1-p^2)$ and  $\gcd\left( \frac{(1- p^2)}{6} , p\right)=1.$ When $r$ runs over a residue system excluding the multiples of $p$, so does $ r \frac{(1- p^2)}{6}.$
		Thus for $p \nmid j,$ \eqref{eq807} can be written as
		\begin{equation}\label{eq808}
			B_{3}\left(p^2n+pj+ \frac{p^2-1}{6}\right) \equiv 0 \pmod3.
		\end{equation}
	Now, we consider the second case, when $p |n.$ Here replacing $n$ by $pn$ in \eqref{eq805}, we get
		\begin{equation}\label{eq809}
			a(p^2n) =  -p \cdot a\left(n\right).
		\end{equation}
		Further, replacing $n$ by $6n+1$ in \eqref{eq809} we obtain
		\begin{equation}\label{eq810}
			a(6p^2n+ p^2) = -p \cdot  a\left(6n+1\right).
		\end{equation}
		Using \eqref{eq803} in \eqref{eq810}, we get
		\begin{equation}\label{eq811}
			B_{3}\left(p^2n+ \frac{p^2-1}{6} \right) = - p \cdot B_{3} \left(n\right).
		\end{equation}
		Let $p_i \geq 5$ be primes such that $p_i \equiv 5 \pmod 6.$ Further note that
		$$ p_1^2 p_2^2 \cdots p_k^2 n + \frac{p_1^2 p_2^2 \cdots p_k^2-1}{6}= p_1^2 \left(p_2^2 \cdots p_k^2 n + \frac{ p_2^2 \cdots p_k^2 -1}{6} \right)+ \frac{p_1^2-1}{6} .$$
		Repeatedly using \eqref{eq811} and \eqref{eq808}, we get
		\begin{align*}
			&	B_{3} \left( p_1^2 p_2^2 \cdots p_k^2 p_{k+1}^2 n + \frac{p_1^2 p_2^2 \cdots p_k^2 p_{k+1}\left(6j+ p_{k+1}\right)-1}{6} \right) \\
			& \equiv (- p_1) \cdot  B_{3} \left(p_2^2 \cdots p_k^2 p_{k+1}^2 n + \frac{ p_2^2 \cdots p_k^2 p_{k+1}\left(6j+ p_{k+1}\right) -1}{6} \right) \equiv \cdots\\
			& \equiv (- 1)^k (p_1 p_2 \cdots p_k)^k B_{3} \left(p_{k+1}^2 n + p_{k+1} j + \frac{p^2_{k+1}-1}{6} \right) \equiv 0 \pmod3
		\end{align*}
		when $j \not \equiv 0 \pmod{ p_{k+1}}.$ This completes the proof of the theorem.
	\end{proof}
	
		\begin{proof}[Proof of Theorem \ref{thm9}] From \eqref{eq805}, we get that for any prime $p \equiv 5 \pmod 6$
		\begin{equation}\label{eq812}
			a(pn)= -p \cdot a\left(\frac{n}{p}\right).
		\end{equation}  Replacing $n$ by $6n+5,$ we obtain
		\begin{equation}\label{eq813}
			a(6pn+5p)= - p \cdot a\left(\frac{6n+5}{p}\right).
		\end{equation}
		Next replacing $n$ by $p^kn+r$ with $p \nmid r$ in \eqref{eq813}, we obtain
		\begin{equation}\label{eq814}
			a\left(6 \left(p^{k+1}n+pr+ \frac{5p-1}{6} \right)+1\right)= (-p) \cdot  a\left(6\left(p^{k-1}n+ \frac{6r+5-p}{6p}\right)+1\right).
		\end{equation}
		Note that $\frac{5p-1}{6} $ and $\frac{6r+5-p}{6p}$ are integers. Using \eqref{eq814} and \eqref{eq803}, we get
		\begin{equation}\label{eq815}
			B_{3}\left(p^{k+1}n+pr+ \frac{5p-1}{6}\right)\equiv  (-p) \cdot B_{3}\left(p^{k-1}n+ \frac{6r+5-p}{6p} \right) \pmod3.
		\end{equation}
	\end{proof}
	
	\begin{proof}[Proof of Corollary \ref{coro4}]Let $p$ be a prime such that $p \equiv 5 \pmod 6.$ Choose a non negative integer $r$ such that $6r+5=p^{2k-1}.$ Substituting $ k$ by $2k-1$ in \eqref{eq815}, we obtain
		\begin{align*}
			B_{3}\left(p^{2k}n+ \frac{p^{2k}-1}{6} \right)&\equiv  (-p)  B_{3}\left(p^{2k-2}n+ \frac{p^{2k-2}-1}{6} \right)\\
			& \equiv \cdots \equiv \left(-p\right)^k  B_{3}(n) \pmod3.
		\end{align*}
	\end{proof}
\begin{proof} [Proof of Theorem \ref{Newmann1}]
Note that
\begin{equation}\label{eq800b*}
	\sum_{n=0}^{\infty} B_{3}(n)q^{n} \equiv  \frac{(q^3;q^3)_{\infty} (q^{3};q^{3})_{\infty}}{(q;q)^2_{\infty}} \equiv (q;q)_{\infty} (q^{3};q^{3})_{\infty} \pmod 3 .
\end{equation}
Define 
\begin{equation}\label{eq50}
	\sum_{n=0}^{\infty} c(n) q^{n}:= f_1f_3= (q;q)_{\infty} (q^{3};q^{3})_{\infty}.
\end{equation}
Using Newman result [\cite{Newmann1959},Theorem $3$] we find that if $p$ is a prime with $p \equiv 1 \pmod 6,$ then
\begin{equation}\label{eq51}
	c\left(pn+ \frac{p-1}{6}\right) = c\left(\frac{p-1}{6}\right)c(n)- (-1)^{\frac{p-1}{2}} \left( \frac{3}{p} \right) c\left(\frac{n}{p}- \frac{p-1}{6p}\right) .
\end{equation}

Observe that, if $p \nmid (6n+1), $ then $\left(\frac{n- \frac{p-1}{6}}{p} \right)$ is not an integer and
\begin{equation}\label{eq52}
	c\left(\frac{n}{p}- \frac{p-1}{6p}\right) =0 .
\end{equation}
From \eqref{eq51} and \eqref{eq52} we get that if $p \nmid (6n+1), $ then
\begin{equation}\label{eq53}
	c\left(pn+ \frac{p-1}{6}\right) = c\left(\frac{p-1}{6}\right)c(n) .
\end{equation}
Thus, if $p \nmid (6n+1) $ and $c\left(\frac{p-1}{6}\right) \equiv 0 \pmod 3,$ then for $n \geq 0,$
\begin{equation}\label{eq54}
	c\left(pn+ \frac{p-1}{6}\right) \equiv 0 \pmod 3 .
\end{equation}

Replacing $n$ by $ pn+ \frac{p-1}{6}$ in \eqref{eq51}, we get

\begin{equation}\label{eq55}
	c\left(p^2n+ \frac{p^2-1}{6}\right) = c\left(\frac{p-1}{6}\right)c\left(pn+ \frac{p-1}{6}\right)- (-1)^{\frac{p-1}{2}} \left( \frac{3}{p} \right) c\left(n\right) .
\end{equation}
Observe that from \eqref{eq55} if $c\left(\frac{p-1}{6}\right) \equiv 0 \pmod 3, $ then for $n \geq 0,$ we get
\begin{equation}\label{eq56}
	c\left(p^2n+ \frac{p^2-1}{6}\right) \equiv - (-1)^{\frac{p-1}{2}} \left( \frac{3}{p} \right) c\left(n\right) \pmod 3 .
\end{equation}
From \eqref{eq56} and by using mathematical induction, we get that if $c\left(\frac{p-1}{6}\right) \equiv 0 \pmod 3, $ then for $n \geq 0, $ and $k \geq 0, $ we get 
\begin{equation}\label{eq57}
	c\left(p^{2k}n+ \frac{p^{2 k}-1}{6}\right) \equiv \left( - (-1)^{\frac{p-1}{2}} \left( \frac{3}{p} \right) \right)^{k} c\left(n\right) \pmod 3 .
\end{equation}
Replacing $n$ by $pn+ \frac{p-1}{6}$ in \eqref{eq57} and using \eqref{eq54}, we obtain that if $p \nmid (6n+1)$ and $c\left(\frac{p-1}{6}\right) \equiv 0 \pmod 3,$ then for $n \geq 0$ and $k \geq 0,$ we get
\begin{equation}\label{eq58}
	c\left(p^{2 k +1}n+ \frac{p^{2 k+1}-1}{6}\right) \equiv 0 \pmod 3 .
\end{equation}
From \eqref{eq50} and \eqref{eq800b*}, we obtain that for $n \geq 0,$
\begin{equation}\label{eq59}
	B_{3}(n) \equiv c(n) \pmod 3 .
\end{equation}
Thus, combining \eqref{eq59} and \eqref{eq58}, we prove Theorem \ref{Newmann1}. \end{proof}

\section{Proof of Congruences for $B_5(n)$}

	\begin{proof}[Proof of Theorem \ref{thm10}] 
		From equation \eqref{eq501}, we have
		\begin{equation}\label{eq900}
			\sum_{n=0}^{\infty} B_{5}(n)q^{n} = \frac{(q^5;q^5)^2_{\infty}}{(q;q)^2_{\infty}}.
		\end{equation}
		From \eqref{eq900} and \eqref{eq800a}, we get
		\begin{equation}\label{eq900b}
			\sum_{n=0}^{\infty} B_{5}(n)q^{n} = \frac{(q^5;q^5)^2_{\infty}}{(q;q)^2_{\infty}} \equiv (q;q)^{8}_{\infty} \pmod 5 
		\end{equation} and hence
			\begin{equation}\label{eq901}
			\sum_{n=0}^{\infty} B_{5}(n)q^{3n+1} \equiv q (q^3;q^3)^{8}_{\infty}  \equiv \eta^8(3z)\pmod  5 .
		\end{equation}
		
		
		Applying Theorem \ref{thm2.3}, we obtain $\eta^8(3z) \in S_4(\Gamma_{0}(9), \left(\frac{3^8}{\bullet}\right)). $ Notice that $\eta^8(3z) $ has a Fourier expansion i.e.
		\begin{equation}\label{eq902}
			\eta^8(3z)= q-8q^4+ 20 q^{7}+ \cdots= \sum_{n=1}^{\infty}  b(n) q^n. 
		\end{equation}
		Thus, $b(n)=0$ if $n \not \equiv 1 \pmod 3,$ for all $n \geq 0.$ Combining \eqref{eq901} and \eqref{eq902} and comparing the coefficient of $q^{3n+1}$, we obtain
		\begin{equation}\label{eq903}
			B_{5}(n) \equiv b(3n+1) \pmod 5.
		\end{equation}
		Note that $ \eta^8(3z)$ is a Hecke eigenform (see \cite{Martin1996}). Thus, we have
		$$\eta^8(3z)|T_p= \sum_{n=1}^{\infty} \left(b(pn)+ p^3 \cdot  \left(\frac{3^8}{p}\right) b\left(\frac{n}{p}\right)\right) q^n = \lambda(p)\sum_{n=1}^{\infty} b(n)q^n.$$ See that $ \left(\frac{3^8}{p}\right) = 1. $ Comparing the coefficients of $q^n$ on both sides of the above equation, we get
		\begin{equation}\label{eq904}
			b(pn)+ p^3 \cdot b\left(\frac{n}{p}\right) = \lambda(p) b(n).
		\end{equation}
		Note that $b(1)=1$ and $b(\frac{1}{p})=0.$ Thus, if we put $n=1$ in \eqref{eq904}, we get $b(p)=\lambda(p).$ Observe that $b(p)=0$ for all $p \not \equiv 1 \pmod 3$ this implies that $\lambda(p)=0$ for all $p \not \equiv 1 \pmod 3.$ From \eqref{eq904} we get that for all $p \not \equiv 1 \pmod 3$, we have 
		\begin{equation}\label{eq905}
			b(pn)+ p^3 \cdot b\left(\frac{n}{p}\right) =0.
		\end{equation}
		Now, we consider two cases here. If $p \not| n,$ then 
		replacing $n$ by $pn+r$ with $\gcd(r,p)=1$ in \eqref{eq905}, we get
		\begin{equation}\label{eq906}
			b(p^2n+ pr)=0 .
		\end{equation}
			Further, replacing $n$ by $3n-pr+1$ in \eqref{eq906} and using \eqref{eq903}, we get 
		\begin{equation}\label{eq907}
			B_{5}\left(p^2n +pr \frac{(1- p^2)}{3} +  \frac{p^2-1}{3}\right) \equiv 0 \pmod5.
		\end{equation}
		Since $p \equiv 2 \pmod 3,$ we have $3| (1-p^2)$ and  $\gcd\left( \frac{(1- p^2)}{3} , p\right)=1.$ When $r$ runs over a residue system excluding the multiples of $p$, so does $ r \frac{(1- p^2)}{3}.$
		Thus for $p \nmid j,$ \eqref{eq907} can be written as
		\begin{equation}\label{eq908}
			B_{5}\left(p^2n+pj+ \frac{p^2-1}{3}\right) \equiv 0 \pmod5.
		\end{equation}
	Now, we consider the second case, when $p |n.$ Here replacing $n$ by $pn$ in \eqref{eq905}, we get
		\begin{equation}\label{eq909}
			b(p^2n) =  -p^3 \cdot b\left(n\right).
		\end{equation}
		Further, substituting $n$ by $3n+1$ in \eqref{eq909} we obtain
		\begin{equation}\label{eq910}
			b(3p^2n+ p^2) = -p^3 \cdot  b\left(3n+1\right).
		\end{equation}
		Combining \eqref{eq903} and \eqref{eq910}, we get
		\begin{equation}\label{eq911}
			B_{5}\left(p^2n+ \frac{p^2-1}{3} \right) \equiv - p^3 \cdot B_{5} \left(n\right) \pmod5.
		\end{equation}
		Let $p_i$ be primes such that $p_i \equiv 2 \pmod 3.$ Further notice that
		$$ p_1^2 p_2^2 \cdots p_k^2 n + \frac{p_1^2 p_2^2 \cdots p_k^2-1}{3}= p_1^2 \left(p_2^2 \cdots p_k^2 n + \frac{ p_2^2 \cdots p_k^2 -1}{3} \right)+ \frac{p_1^2-1}{3} .$$
		Repeatedly using \eqref{eq911} and \eqref{eq908}, we get
		\begin{align*}
			&	B_{5} \left( p_1^2 p_2^2 \cdots p_k^2 p_{k+1}^2 n + \frac{p_1^2 p_2^2 \cdots p_k^2 p_{k+1}\left(3j+ p_{k+1}\right)-1}{3} \right) \\
			& \equiv (- p_1^3) \cdot  B_{5} \left(p_2^2 \cdots p_k^2 p_{k+1}^2 n + \frac{ p_2^2 \cdots p_k^2 p_{k+1}\left(3j+ p_{k+1}\right) -1}{3} \right) \equiv \cdots\\
			& \equiv (- 1)^k (p_1p_2\ldots p_k)^3 \cdot B_{5} \left(p_{k+1}^2 n + p_{k+1} j + \frac{p^2_{k+1}-1}{3} \right) \equiv 0 \pmod5
		\end{align*}
		when $j \not \equiv 0 \pmod{ p_{k+1}}.$ Thus, the theorem is proved.
	\end{proof}
	
		\begin{proof}[Proof of Theorem \ref{thm11}] From \eqref{eq905}, we get that for any prime $p \equiv 2 \pmod 3$
		\begin{equation}\label{eq912}
			b(pn)= -p^3 \cdot b\left(\frac{n}{p}\right).
		\end{equation}  Replacing $n$ by $3n+2,$ we obtain
		\begin{equation}\label{eq913}
			b(3pn+2p)= - p^3 \cdot b\left(\frac{3n+2}{p}\right).
		\end{equation}
		Further replacing $n$ by $p^kn+r$ with $p \nmid r$ in \eqref{eq913}, we obtain
		\begin{equation}\label{eq914}
			b\left(3 \left(p^{k+1}n+pr+ \frac{2p-1}{3} \right)+1\right)= (-p^3) \cdot  b\left(3\left(p^{k-1}n+ \frac{3r+2-p}{3p}\right)+1\right).
		\end{equation}
		Note that $\frac{2p-1}{3} $ and $\frac{3r+2-p}{3p}$ are integers. Combining \eqref{eq914} and \eqref{eq903}, we get
		\begin{equation}\label{eq915}
			B_{5}\left(p^{k+1}n+pr+ \frac{2p-1}{3}\right)\equiv  (-p^3)  B_{5}\left(p^{k-1}n+ \frac{3r+2-p}{3p} \right) \pmod5.
		\end{equation}
	\end{proof}
	
	\begin{proof}[Proof of Corollary \ref{coro6}]Let $p$ be a prime such that $p \equiv 2 \pmod 3.$ Choose a non negative integer $r$ such that $3r+2=p^{2k-1}.$ Replacing $ k$ by $2k-1$ in \eqref{eq915}, we obtain
		\begin{align*}
			B_{5}\left(p^{2k}n+ \frac{p^{2k}-1}{3} \right)&\equiv  (-p^3)  B_{5}\left(p^{2k-2}n+ \frac{p^{2k-2}-1}{3} \right)\\
			& \equiv \cdots \equiv \left(-p^3\right)^k  B_{5}(n) \pmod5.
		\end{align*}
	\end{proof}

\begin{proof}[Proof of Theorem \ref{Newmann2}]
Observe that
\begin{equation}\label{eq1800b*}
	\sum_{n=0}^{\infty} B_{5}(n)q^{n} \equiv  \frac{(q^5;q^5)_{\infty} (q^{5};q^{5})_{\infty}}{(q;q)^2_{\infty}} \equiv (q;q)^3_{\infty} (q^{5};q^{5})_{\infty} \pmod 5 .
\end{equation}
Define 
\begin{equation}\label{eq150}
	\sum_{n=0}^{\infty} t(n) q^{n}:= f_1^3f_5= (q;q)^3_{\infty} (q^{5};q^{5})_{\infty}.
\end{equation}
Using  a result of Newman [\cite{Newmann1959},Theorem $3$] we get that if $p$ is a prime with $p \equiv 1 \pmod 6,$ then
\begin{equation}\label{eq151}
	t\left(pn+ \frac{p-1}{3}\right) = t\left(\frac{p-1}{3}\right)t(n)- \left( \frac{5}{p} \right) \cdot p \cdot  t\left(\frac{n}{p}- \frac{p-1}{3p}\right) .
\end{equation}

Notice that, if $p \nmid (3n+1), $ then $\left(\frac{n- \frac{p-1}{3}}{p} \right)$ is not an integer and hence
\begin{equation}\label{eq152}
	t\left(\frac{n}{p}- \frac{p-1}{3p}\right) =0 .
\end{equation}
Combining \eqref{eq151} and \eqref{eq152}, we see that if $p \nmid (3n+1), $ then
\begin{equation}\label{eq153}
	t\left(pn+ \frac{p-1}{3}\right) = t\left(\frac{p-1}{3}\right)t(n) .
\end{equation}
Thus, if $p \nmid (3n+1) $ and $t\left(\frac{p-1}{3}\right) \equiv 0 \pmod 5,$ then for $n \geq 0,$
\begin{equation}\label{eq154}
	t\left(pn+ \frac{p-1}{3}\right) \equiv 0 \pmod 5 .
\end{equation}
Replacing $n$ by $ pn+ \frac{p-1}{3}$ in \eqref{eq151}, we get

\begin{equation}\label{eq155}
	t\left(p^2n+ \frac{p^2-1}{3}\right) = t\left(\frac{p-1}{3}\right)t\left(pn+ \frac{p-1}{3}\right)- \left( \frac{5}{p} \right) \cdot p \cdot t\left(n\right) .
\end{equation}
Notice that from \eqref{eq155} if $t\left(\frac{p-1}{3}\right) \equiv 0 \pmod 3, $ then for $n \geq 0,$ we get
\begin{equation}\label{eq156}
	t\left(p^2n+ \frac{p^2-1}{3}\right) \equiv - \left( \frac{5}{p} \right) \cdot p \cdot t\left(n\right) \pmod 5 .
\end{equation}
From \eqref{eq156} and by using mathematical induction, we get that if $t\left(\frac{p-1}{3}\right) \equiv 0 \pmod 5, $ then for $n \geq 0, $ and $k \geq 0, $ we get 
\begin{equation}\label{eq157}
	t\left(p^{2 k}n+ \frac{p^{2k}-1}{3}\right) \equiv \left( -  \left( \frac{5}{p}  \right)\cdot p \right)^{k} t\left(n\right) \pmod 5 .
\end{equation}
Replacing $n$ by $pn+ \frac{p-1}{3}$ in \eqref{eq157} and using \eqref{eq154}, we obtain that if $p \nmid (3n+1)$ and $t\left(\frac{p-1}{3}\right) \equiv 0 \pmod 3,$ then for $n \geq 0$ and $k \geq 0,$ we get
\begin{equation}\label{eq158}
	t\left(p^{2 k +1}n+ \frac{p^{2k+1}-1}{3}\right) \equiv 0 \pmod 5 .
\end{equation}
From \eqref{eq150} and \eqref{eq1800b*}, we obtain that for $n \geq 0,$
\begin{equation}\label{eq159}
	B_{5}(n) \equiv t(n) \pmod 5 .
\end{equation}
Thus, combining \eqref{eq159} and \eqref{eq158}, we get the result of Theorem \ref{Newmann2}.
\end{proof}

\section{Proof of Theorems \ref{mainthmRaduSellers} [Radu Seller's Method]}
\subsection{An algorithmic approach by Radu and Sellers}
We begin with recalling an algorithm developed by Radu and Sellers \cite{Radu2011} that will be used to prove Theorem \ref{mainthmRaduSellers}. Let $M$ be a positive integer and let $R(M)$ denote the set of integers sequences $r=(r_\delta)_{\delta|M}$ indexed by the positive divisors of $M$. For $r \in R(M)$ and the positive divisors $1=\delta_1<\delta_2<\cdots<\delta_{i_M}=M$ of $M$,  we set $r=(r_{\delta_1},r_{\delta_2},\dots , r_{\delta_{i_M}})$. We define $c_r(n)$  by
\begin{align*}
	\sum\limits_{n=0}^{\infty} c_r(n)q^n:=\prod\limits_{\delta|M}(q^{\delta};q^{\delta})_{\infty}^{r_{\delta}}=\prod\limits_{\delta|M}\prod\limits_{n=1}^{\infty}(1-q^{n\delta})^{r_{\delta}}.
\end{align*}
Radu and Sellers \cite{Radu2011} approach to prove congruences for $c_r(n)$ modulo a positive integer reduced the number of cases that we need to check as compared with the classical method which uses Sturm's bound alone.

Let $m\geq0$ and $s$ be integers. We denote by $[s]_m$ the residue class of $s$ in $\mathbb{Z}_m$ and we denote by $\mathbb{S}_m$ the set of squares in $\mathbb{Z}_m^{*}$. For $t \in \{0,1, \dots, m-1\}$ and $r\in R(M)$, the subset $P_{m,r}(t)\subseteq \{0,1, \dots, m-1\}$ is defined as
\begin{align*}
	P_{m,r}(t):=\left\{t^{'}: \exists [s]_{24m} \textrm{ such that } t^{'}\equiv ts+\frac{s-1}{24} \sum\limits_{\delta|M} \delta r_{\delta} \pmod m\right\}.
\end{align*}

\begin{definition}
	For positive integers $m$, $M$ and $N$,  let $r=(r_\delta)\in R(M)$ and $t \in \{0,1, \dots, m-1\}$. Let $\kappa=\kappa(m):=\gcd(m^2-1,24)$ and write
	\begin{align*}
		\prod\limits_{\delta|M} \delta^{|r_{\delta}|}=2^s\cdot j,
	\end{align*}
	where $s$ and $j$ are non-negative integers with $j$ odd. The set $\Delta^{*}$ is the collection of all tuples $(m, M, N, (r_{\delta}), t)$ satisfying the following conditions.
	\begin{itemize}
		\item[(a)] Every prime divisor of $m$ is also a divisor of $N$.
		\item[(b)] If $\delta|M,$ then $\delta|mN$ for every $\delta\geqslant 1$ such that $r_{\delta} \ne 0$.
		\item[(c)] $\kappa N \sum\limits_{\delta|M} r_{\delta}mN/\delta \equiv 0 \pmod {24}$.
		\item[(d)] $\kappa N \sum\limits_{\delta|M} r_{\delta} \equiv 0 \pmod {8}$.
		\item[(e)] $\frac{24m}{\gcd(-24\kappa t-\kappa\sum\limits_{\delta|M} \delta r_{\delta}, 24m)}$divides $N$.
		\item[(f)] If $2|m$, then either ($4|\kappa N$ and $8|sN$) or ($2|s$ and $8|(1-j)N$).
	\end{itemize}
\end{definition}

For positive integers $m$, $M$ and $N$, $\gamma= \begin{bmatrix}
	a &b\\ c&d	\end{bmatrix} \in \Gamma$, $r\in R(M)$ and $a = (a_{\delta})\in R(N)$, we define
\begin{align*}
	p_{m,r}(\gamma):= \min \limits_{\lambda\in\{0,1,\dots,m-1\}} \frac{1}{24} \sum \limits_{\delta|M} r_{\delta}\frac{\gcd^2(\delta a + \delta \kappa \lambda c,mc)}{\delta m}
\end{align*}
and
\begin{align*}
	p^{*}_{a}(\gamma):=\frac{1}{24} \sum \limits_{\delta|N} a_{\delta}\frac{\gcd^2(\delta,c)}{\delta}.
\end{align*}

The following lemma is given by Radu \cite[Lemma $4.5$]{radu2009}.
\begin{lemma}\label{finite check}
	Let $u$ be a positive integer, $(m, M, N, (r_{\delta}), t)\in \Delta^{*}$ and $a=(a_{\delta})\in R(N)$. Let $\{\gamma_1, \gamma_2, \dots, \gamma_n\} \subseteq \Gamma$ denote a complete set of representatives of the double cosets of $\Gamma_0(N)\backslash \Gamma / \Gamma_{\infty}$. Assume that $p_{m,r}(\gamma_i) +p^{*}_{a}(\gamma_i)\geq 0$ for all $1 \leq i \leq n$. Let $t_{\min}= \min_{t^{'} \in P_{m,r}(t)} t^{'}$ and
	\begin{align*}
		\nu:= \frac{1}{24} \left\{\left(\sum \limits_{\delta|M} r_{\delta}+\sum \limits_{\delta|N} a_{\delta}\right) [\Gamma: \Gamma_0(N)]-\sum \limits_{\delta|N} \delta a_{\delta} \right\} - \frac{1}{24m}\sum \limits_{\delta|M} \delta r_{\delta} -\frac{t_{\min}}{m}.
	\end{align*}
	If the congruence $c_r(mn+t^{'}) \equiv 0 \pmod u$ holds for all $t^{'} \in P_{m,r}(t)$ and $ 0\leq n \leq \lfloor \nu \rfloor$, then $c_r(mn+t^{'}) \equiv 0 \pmod u$ holds for all $t^{'} \in P_{m,r}(t)$ and $n \geq 0$.
\end{lemma}

The next lemma is given by Wang \cite[Lemma $4.3$]{wang2017}. This result gives the complete set of representatives of the double cosets in  $\Gamma_0(N)\backslash \Gamma / \Gamma_{\infty}$ when $N$ or $\frac{N}{2}$ is a square-free integer.

\begin{lemma}\label{square-free}
	If $N$ or $\frac{N}{2}$ is a square-free integer, then
	\begin{align*}
		\bigcup_{\delta|N} \Gamma_0(N) \begin{bmatrix}
			1 &0\\ \delta &1	\end{bmatrix} \Gamma_{\infty} =\Gamma.
	\end{align*}
\end{lemma}

\subsection{Proof of Theorem \ref{mainthmRaduSellers}}
For an integer $m$, a prime $p\geqslant 5$ and $t\in \{0,1,\dots,m-1\}$, we recall that
\begin{align*}
	\kappa:=&\kappa(m)=\gcd(m^2-1,24),\\
	\hat {p}:=&\frac{p^2-1}{24},\\
	A_t:=&A_t(m,p)=\frac{12m}{\gcd(-\kappa(12 t+p-1),12m)}
	\\
	\epsilon_p:=&\epsilon_p(m,p)=\begin{cases}
		1 & \textrm{ if } p\nmid m,\\
		0 & \textrm{ if } p| m.
	\end{cases}
\end{align*}

We now prove the following three results specific to the proof of Theorem \ref{mainthmRaduSellers}.

\begin{lemma} \label{lemdelta}
	Let $p\geq 3$ be a prime number. For a positive integer $g$, let $e_1, \ldots, e_g$ be positive integers and let $p_1,p_2, \ldots, p_{g}$ be prime numbers. Let 
	$$(m,M,N,r,t)=(p_1^{e_1} p_2^{e_2} \cdots p_g^{e_g}, p,  p^{\epsilon_p} p_1 p_2 \cdots p_{g},r=(r_1=-2,r_p=2),t)$$ 
	where $t \in \{0,1,\ldots, m-1\}$ is such that $A_t|N$. Then $(m,M,N,r,t) \in \Delta^{*}$.
\end{lemma}
\begin{proof}
	It is immediate that the conditions (a) and (b) in the definition of $\Delta^{*}$ are satisfied. Since $M=p$ and $r=(r_1=-2,r_p=2)$, we see that 
	\begin{align*}
		\kappa N \sum_{\delta|M} r_{\delta}mN/\delta=2m\kappa (1-p) \frac {N^2}{p}
	\end{align*}
 where $\frac {N}{p}$ is an integer. Since $p$ is odd, we have $ 2(1-p)\equiv 0 \pmod {4}$. To prove that (c) holds, it is enough to show that $ \kappa m\equiv 0 \pmod {6}$. Indeed, we note that
	\begin{align*}
		\gcd(m^2-1,8)=\begin{cases}
			8 & \textrm{ if } $m$ \textrm{ is odd,}\\ 
			1 & \textrm{ if } $m$ \textrm{ is even}\\ 
		\end{cases}
		\quad\textrm{and}\quad
		\gcd(m^2-1,3)=\begin{cases}
			3 & \textrm{ if } 3 \nmid m,\\ 
			1 & \textrm{ if } 3|m.\\ 
		\end{cases}
	\end{align*}
	This implies that
	\begin{align*}
		\kappa=\gcd(m^2-1,24)=\begin{cases}
			24& \textrm{ if } \gcd(m,6)=1,\\ 
			8 & \textrm{ if } \gcd(m,6)=3,\\ 
			3 & \textrm{ if } \gcd(m,6)=2,\\ 
			1 & \textrm{ if } \gcd(m,6)=6.\\ 
		\end{cases}
	\end{align*}
	Therefore, it can be easily seen that $\kappa m \equiv 0 \pmod 6$ and that (c) holds. 
	
	 Next, we note that $\sum_{\delta|M} r_{\delta} =0$ and therefore (d) is also satisfied. From the conditions on $t$, we see that (e) is also satisfied. For (f), we see that $\prod_{\delta|M} \delta^{|r_{\delta}|}=p^2$ gives $s=0$ and $j=p^2$. Now $(f)$ follows quickly since $p^2\equiv 1 \pmod 8$. This completes the proof. 
\end{proof}

\begin{lemma} \label{pmrpositive}
	Let $p\geq 3$ be a prime number. For a positive integer $m$, $M=p$, $r=(r_1=-2,r_p=2)$ and $\gamma= \begin{bmatrix}
		a &b\\ c&d	\end{bmatrix} \in \Gamma$, we have
	\begin{align*}
		p_{m,r}(\gamma)&= \min \limits_{\lambda\in\{0,1,\dots,m-1\}} \frac{1}{24} \sum \limits_{\delta|M} r_{\delta}\frac{\gcd^2(\delta a + \delta \kappa \lambda c,mc)}{\delta m}\geqslant	-\frac{(p-1)m}{12p}\\
	\end{align*}
\end{lemma}
\begin{proof} We note that
	\begin{align*}
		\sum \limits_{\delta|M} r_{\delta}\frac{\gcd^2(\delta a + \delta \kappa \lambda c,mc)}{\delta m}&=2\frac{\gcd^2(p a + p \kappa \lambda c,mc)}{p m}
		-2\frac{\gcd^2( a + \kappa \lambda c,mc)}{ m}\\
		&=\frac{2( \gcd^{2} (p a + p \kappa \lambda c,mc)
			-p \gcd^2( a + \kappa \lambda c,mc))}{pm} 
	\end{align*}
	Since $\gamma \in \Gamma=SL_2(\Z)$, we have $ad-bc=1$ which implies that $\gcd(a,c)=1$. Thus it follows that $\gcd( a +  \kappa \lambda c,c)=1$. Therefore it is enough to prove that
	\begin{align*}
		G:={\gcd}^2(p a + p \kappa \lambda c,mc)
		-p{\gcd}^2( a + \kappa \lambda c,m)\geq -(p-1)m^2
	\end{align*}
	for each $\lambda\in\{0,1,\dots,m-1\}$. 
	
	Let $\lambda\in\{0,1,\dots,m-1\}$ be fixed. We consider the two cases $p|c$ and $p\nmid c$ separately.
	
	\noindent\textbf{Case 1: $p|c$.} We set $c_p=\frac{c}{p}$. We observe that $\gcd(a,c)=1$ implies that $\gcd(a,c_p)=1$ which in turn implies $\gcd( a +  \kappa \lambda c,c_p)=\gcd( a +  \kappa \lambda p c_p,c_p)=1$. Hence we have 
	\begin{align*}
		G&=p^2{\gcd}^2( a +  \kappa \lambda c,mc_p)
		-p{\gcd}^2( a + \kappa \lambda c,m)\\
		&=p^2{\gcd}^2( a +  \kappa \lambda c,m)
		-p{\gcd}^2( a + \kappa \lambda c,m)>0.
	\end{align*}
	\textbf{Case 2: $p\nmid c$.} In this case, we have
	\begin{align*}
		G&={\gcd}^2( p(a +  \kappa \lambda c),m)
		-p{\gcd}^2( a + \kappa \lambda c,m).
	\end{align*}
	If $p\nmid m$, then $\gcd( p(a +  \kappa \lambda c),m)=\gcd( a +  \kappa \lambda c,m)$ and thus $$G=-(p-1)\gcd^2( a +  \kappa \lambda c,m)\geq -(p-1)m^2.$$
	We now assume that $p|m$. We set $m_p=\frac{m}{p}\in \Z$. Then
	\begin{align*}
		G&=p^2{\gcd}^2( a +  \kappa \lambda c,m_p)
		-p{\gcd}^2( a + \kappa \lambda c,m).
	\end{align*}
	Let $d=\gcd( a + \kappa \lambda c,m)$. Let ord$_p(n)$ denote the highest exponent of $p$ dividing a positive integer $n$. It is clear that ord$_p(d)\leq$ord$_p(m)=$ord$_p(m_p)+1$. If ord$_p(d)\leq$ord$_p(m_p)$, then $G=d^2(p^2-p)>0$. If ord$_p(d)=$ord$_p(m_p)+1$, then $\gcd( a + \kappa \lambda c,m_p)=\frac{d}{p}$ and therefore $G=p^2(\frac{d}{p})^2-pd^2=(1-p)d^2\geq (1-p)m^2$. This completes the proof.
\end{proof}
\begin{lemma} \label{pmrpapositive}
	Let $p\geq 3$ be a prime number. Let $p\geq 3$ be a prime number. For a positive integer $g$, let $e_1, \ldots, e_g$ be positive integers and let $p_1,p_2, \ldots, p_{g}$ be prime numbers. Let  
	$$(m,M,N,r)=(p_1^{e_1} p_2^{e_2} \cdots p_g^{e_g}, p,  p^{\epsilon_p} p_1 p_2 \cdots p_{g},r=(r_1=-2,r_p=2)).$$
	Let $a=(a_{\delta})\in R(N)$ such that 
	\begin{align*}
		a_{\delta}=\begin{cases}
			2m(p-1) & \textrm{ if } \delta=p\\
			0 & \textrm{ if } \delta|N \textrm{ with } \delta\ne p.
		\end{cases}
	\end{align*}
	Let $\gamma_{\delta}= \begin{bmatrix}
		1 & 0 \\ \delta&1	\end{bmatrix}$ for $\delta|N$. Then we have
	\begin{align*}
		p_{m,r}(\gamma_{\delta})+p^{*}_{a}(\gamma_{\delta})\geqslant 0.
	\end{align*}
\end{lemma}
\begin{proof} 
	For $\delta|N$, using Lemma \ref{pmrpositive} and the definition of $a=(a_{\delta})$, we see that
	\begin{align*}
		p_{m,r}(\gamma_{\delta})+p^{*}_{a}(\gamma_{\delta}) &\geqslant -\frac{m(p-1)}{12p}+ \frac{m(p-1){\gcd}^2(p,\delta)}{12p}\\
		&=\frac{m(p-1)}{12p}\left( {\gcd}^2(p,\delta)- 1 \right)\\
		&\geq 0.
	\end{align*}
\end{proof}

\begin{lemma}\label{lembound}
	Let $p\geq 3$ be a prime and let $u$ be an integer. Let $(m,M,N,r,t)$ be as defined in Lemma \ref{lemdelta}. Let $t_{\min}= \min_{t^{'} \in P_{m,r}(t)} t^{'}$.	If the congruence $a_p(mn+t^{'}) \equiv 0 \pmod u$ holds for all $t^{'} \in P_{m,r}(t)$ and $ 0\leq n \leq \left \lfloor \frac{(p-1)m\left[(p+1)^{\epsilon_p}(p_1+1)(p_2+1)\cdots (p_{g}+1)-p\right]}{12} - \frac{p-1}{12m} \right \rfloor$, then $a_p(mn+t^{'}) \equiv 0 \pmod u$ holds for all $t^{'} \in P_{m,r}(t)$ and $n \geq 0$.
\end{lemma}
\begin{proof}
	It is enough to show that the assumptions of Lemma \ref{finite check} are satisfied and that the upper bound in Lemma \ref{finite check} is less than or equal to $\left \lfloor \frac{(p-1)m\left[(p+1)^{\epsilon_p}(p_1+1)(p_2+1)\cdots (p_{g}+1)-p\right]}{12} - \frac{p-1}{12m} \right \rfloor$. For $\delta|N$, we set $\gamma_{\delta}=\begin{bmatrix} 1 &0\\ \delta&1 \end{bmatrix}$. Since $\epsilon_2= 0$ or $1$, $N$ or $\frac{N}{2 }$ is a square-free integer. Thus Lemma \ref{square-free} implies that $\{\gamma_{\delta}: \delta|N\}$ forms a complete set of double coset representatives of $\Gamma_0(N)\backslash \Gamma / \Gamma_{\infty}$. Lemma \ref{pmrpapositive} implies that $p_{m,r}(\gamma_{\delta})\geq 0$ for each $\delta|N$.	Therefore we take $a_{\delta}=0$ for each $\delta|N$, that is, $a=(0,0, \ldots, 0)\in R(N)$ and hence $p_{m,r}(\gamma_{\delta}) +p_{a}^{*}(\gamma_{\delta})\geq0$ for each $\delta|N$. Since $t_{\min}\geq 0$, we have
	\begin{align*}
		\lfloor \nu \rfloor&= \left \lfloor \frac{1}{24}\left(2(p-1)m[\Gamma:\Gamma_0(N)] - 2p(p-1)m \right) - \frac{p-1}{12m} -\frac{t_{\min}}{m} \right \rfloor\\
		& \leqslant \left \lfloor \frac{(p-1)m}{12}\left\{[\Gamma:\Gamma_0(N)] - p\right\} -\frac{p-1}{12m} \right \rfloor\\
		&= \left \lfloor \frac{(p-1)m\left[(p+1)^{\epsilon_p}(p_1+1)(p_2+1)\cdots (p_{g}+1)-p\right]}{12} - \frac{p-1}{12m} \right \rfloor.
	\end{align*}
\end{proof}

\begin{proof}[Proof of Theorem \ref{mainthmRaduSellers}]
	We get from \eqref{eq501} that
	\begin{align*}
		\sum\limits_{n=0}^{\infty} B_p(n) q^n= \frac{(q^p;q^p)^{2}_{\infty}}{(q;q)^2_{\infty}}.
	\end{align*}
	Let $(m,M,N,r,t)=(p_1^{e_1} p_2^{e_2} \cdots p_g^{e_g},p, p^{\epsilon_p} p_1 p_2 \cdots p_{g},r=(r_1=-2,r_p=2),t)$ be such that $t \in \{0,1,\ldots, m-1\}$ and $A_t|N$. Then Lemma \ref{lemdelta} implies $(m,M,N,r,t) \in \Delta^{*}$. Since $\sum_{\delta|M}\delta r_{\delta}=2(p-1)$, we see that $P_{m,r}(t)=P(t)$. Now the assertion follows from Lemma \ref{lembound}.
\end{proof}

\section{Conclusion}
	By Radu and Seller's method, for a prime $p\geq3$, we obtain an algorithm in Theorem \ref{mainthmRaduSellers} for congruences of the type $B_p(n) \equiv 0 \pmod u$ when $m=p_1^{e_1} p_2^{e_2}\cdots p_g^{e_g}$ where $p_i\geq 5$ are prime numbers and $e_1, e_2, \cdots, e_g$ are positive integers. Further, we obtain several infinite families of congruences for $B_3(n)$ and $B_5(n)$ by using Hecke eigen forms and Newman Identity. Since we have obtained the density results for divisibility of $B_{\ell}(n)$ by powers of primes dividing $\ell$,
	it will be interesting to see more algebraic results in the same direction.

%
	\noindent{\bf Data availability statement:} There is no data associated to our manuscript.
	
	\noindent{\bf Conflict of interest:} There is no conflict of interest between the two authors.
	
	\noindent{\bf Acknowledgement.} The second author is grateful to the National Board for Higher Mathematics, Department of Atomic Energy, India for postdoctoral fellowship.

\end{document}